\documentclass[10pt,letter,oneside]{amsart} 

\usepackage[voffset = 0pt, hoffset = 0pt, footskip = 20pt,  textheight = 650pt, textwidth = 480pt]{geometry}

\usepackage{amsmath}
\usepackage{amsfonts}
\usepackage{amssymb}
\usepackage{mathtools}
\usepackage{physics}
\usepackage{mathrsfs} 
\usepackage{amsthm} 
\usepackage{tikz}

\usepackage{verbatim}

\usepackage{amsthm}
\numberwithin{equation}{subsection}
\theoremstyle{plain}
\newtheorem{thm}{Theorem}[section]
\newtheorem{prop}[thm]{Proposition}
\newtheorem{lem}[thm]{Lemma}

\theoremstyle{definition}
\newtheorem{defn}[thm]{Definition}
\theoremstyle{remark}
\newtheorem{rmk}[thm]{Remark}
\newtheorem{exmp}[thm]{Example}

\numberwithin{figure}{section}

\begin{document}



\newcommand{\id}{\mathrm{id}}
\newcommand{\into}{\hookrightarrow}
\newcommand{\onto}{\twoheadrightarrow}
\newcommand{\otspam}{\mathbin{\reflectbox{$\mapsto$}}}

\newcommand{\DD}{\mathbb{D}}
\newcommand{\fbar}{\overline{f}}
\newcommand{\gobar}{\overline{\omega}}


\newcommand{\e}{\mathrm{e}}
\newcommand{\xbar}{\bar{x}}
\newcommand{\ybar}{\bar{y}}
\newcommand{\vp}{\mathbf{p}}
\newcommand{\vw}{\mathbold{w}}
\newcommand{\vx}{\mathbf{x}}
\newcommand{\vy}{\mathbf{y}}
\newcommand{\vga}{\mathbold{\alpha}}
\newcommand{\vgb}{\mathbold{\beta}}
\newcommand{\vxi}{\mathbold{\xi}}
\newcommand{\im}{\operatorname{im}}
\newcommand{\coker}{\operatorname{coker}}
\newcommand{\Hom}{\operatorname{Hom}}
\newcommand{\End}{\operatorname{End}}
\newcommand{\res}{\operatorname{res}}
\newcommand{\Ann}{\operatorname{Ann}}
\newcommand{\Spec}{\operatorname{Spec}}
\newcommand{\Proj}{\operatorname{Proj}}
\newcommand{\oo}{\mathscr{O}}
\newcommand{\Pic}{\operatorname{Pic}}
\newcommand{\acton}{\mathbin{\rotatebox[origin=c]{-90}{$\circlearrowleft$}}}
\newcommand{\actedonby}{\mathbin{\rotatebox[origin=c]{90}{$\circlearrowright$}}}
\newcommand{\scr}[1]{\mathscr{#1}}
\newcommand{\obj}{\operatorname{obj}}
\newcommand{\Ad}{\operatorname{Ad}}
\newcommand{\ad}{\operatorname{ad}}
\newcommand{\LgrpGL}{\mathrm{GL}}
\newcommand{\LgrpSL}{\mathrm{SL}}
\newcommand{\LgrpO}{\mathrm{O}}
\newcommand{\LgrpSO}{\mathrm{SO}}
\newcommand{\LgrpSp}{\mathrm{Sp}}
\newcommand{\LgrpU}{\mathrm{U}}
\newcommand{\LgrpSU}{\mathrm{SU}}
\newcommand{\LgrpPGL}{\mathrm{PGL}}
\newcommand{\LgrpPSL}{\mathrm{PSL}}


\newcommand{\ga}{\alpha} 
\newcommand{\gb}{\beta} 
\newcommand{\gc}{\gamma} \newcommand{\gC}{\Gamma}
\newcommand{\gd}{\delta} \newcommand{\gD}{\Delta}
\newcommand{\gve}{\varepsilon} \renewcommand{\ge}{\epsilon} 
\newcommand{\gi}{\iota} 
\newcommand{\gk}{\kappa} 
\newcommand{\gl}{\lambda} \newcommand{\gL}{\Lambda}
\newcommand{\go}{\omega} \newcommand{\gO}{\Omega}
\newcommand{\gvp}{\varphi} 
\newcommand{\gr}{\rho} 
\newcommand{\gs}{\sigma} \newcommand{\gS}{\Sigma}
\newcommand{\gth}{\theta} \newcommand{\gTh}{\Theta}
\newcommand{\gu}{\upsilon} 
\newcommand{\gz}{\zeta} 

\newcommand{\bbA}{\mathbb{A}}
\newcommand{\bbB}{\mathbb{B}}
\newcommand{\bbC}{\mathbb{C}}
\newcommand{\bbD}{\mathbb{D}}
\newcommand{\bbE}{\mathbb{E}}
\newcommand{\bbF}{\mathbb{F}}
\newcommand{\bbG}{\mathbb{G}}
\newcommand{\bbH}{\mathbb{H}}
\newcommand{\bbI}{\mathbb{I}}
\newcommand{\bbJ}{\mathbb{J}}
\newcommand{\bbK}{\mathbb{K}}
\newcommand{\bbL}{\mathbb{L}}
\newcommand{\bbM}{\mathbb{M}}
\newcommand{\bbN}{\mathbb{N}}
\newcommand{\bbO}{\mathbb{O}}
\newcommand{\bbP}{\mathbb{P}}
\newcommand{\bbQ}{\mathbb{Q}}
\newcommand{\bbR}{\mathbb{R}}
\newcommand{\bbS}{\mathbb{S}}
\newcommand{\bbT}{\mathbb{T}}
\newcommand{\bbU}{\mathbb{U}}
\newcommand{\bbV}{\mathbb{V}}
\newcommand{\bbW}{\mathbb{W}}
\newcommand{\bbX}{\mathbb{X}}
\newcommand{\bbY}{\mathbb{Y}}
\newcommand{\bbZ}{\mathbb{Z}}

\newcommand{\calA}{\mathcal{A}}
\newcommand{\calB}{\mathcal{B}}
\newcommand{\calC}{\mathcal{C}}
\newcommand{\calD}{\mathcal{D}}
\newcommand{\calE}{\mathcal{E}}
\newcommand{\calF}{\mathcal{F}}
\newcommand{\calG}{\mathcal{G}}
\newcommand{\calH}{\mathcal{H}}
\newcommand{\calI}{\mathcal{I}}
\newcommand{\calJ}{\mathcal{J}}
\newcommand{\calK}{\mathcal{K}}
\newcommand{\calL}{\mathcal{L}}
\newcommand{\calM}{\mathcal{M}}
\newcommand{\calN}{\mathcal{N}}
\newcommand{\calO}{\mathcal{O}}
\newcommand{\calP}{\mathcal{P}}
\newcommand{\calQ}{\mathcal{Q}}
\newcommand{\calR}{\mathcal{R}}
\newcommand{\calS}{\mathcal{S}}
\newcommand{\calT}{\mathcal{T}}
\newcommand{\calU}{\mathcal{U}}
\newcommand{\calV}{\mathcal{V}}
\newcommand{\calW}{\mathcal{W}}
\newcommand{\calX}{\mathcal{X}}
\newcommand{\calY}{\mathcal{Y}}
\newcommand{\calZ}{\mathcal{Z}}

\newcommand{\scrA}{\mathscr{A}}
\newcommand{\scrB}{\mathscr{B}}
\newcommand{\scrC}{\mathscr{C}}
\newcommand{\scrD}{\mathscr{D}}
\newcommand{\scrE}{\mathscr{E}}
\newcommand{\scrF}{\mathscr{F}}
\newcommand{\scrG}{\mathscr{G}}
\newcommand{\scrH}{\mathscr{H}}
\newcommand{\scrI}{\mathscr{I}}
\newcommand{\scrJ}{\mathscr{J}}
\newcommand{\scrK}{\mathscr{K}}
\newcommand{\scrL}{\mathscr{L}}
\newcommand{\scrM}{\mathscr{M}}
\newcommand{\scrN}{\mathscr{N}}
\newcommand{\scrO}{\mathscr{O}}
\newcommand{\scrP}{\mathscr{P}}
\newcommand{\scrQ}{\mathscr{Q}}
\newcommand{\scrR}{\mathscr{R}}
\newcommand{\scrS}{\mathscr{S}}
\newcommand{\scrT}{\mathscr{T}}
\newcommand{\scrU}{\mathscr{U}}
\newcommand{\scrV}{\mathscr{V}}
\newcommand{\scrW}{\mathscr{W}}
\newcommand{\scrX}{\mathscr{X}}
\newcommand{\scrY}{\mathscr{Y}}
\newcommand{\scrZ}{\mathscr{Z}}

\newcommand{\frA}{\mathfrak{A}}
\newcommand{\frB}{\mathfrak{B}}
\newcommand{\frC}{\mathfrak{C}}
\newcommand{\frD}{\mathfrak{D}}
\newcommand{\frE}{\mathfrak{E}}
\newcommand{\frF}{\mathfrak{F}}
\newcommand{\frG}{\mathfrak{G}}
\newcommand{\frH}{\mathfrak{H}}
\newcommand{\frI}{\mathfrak{I}}
\newcommand{\frJ}{\mathfrak{J}}
\newcommand{\frK}{\mathfrak{K}}
\newcommand{\frL}{\mathfrak{L}}
\newcommand{\frM}{\mathfrak{M}}
\newcommand{\frN}{\mathfrak{N}}
\newcommand{\frO}{\mathfrak{O}}
\newcommand{\frP}{\mathfrak{P}}
\newcommand{\frQ}{\mathfrak{Q}}
\newcommand{\frR}{\mathfrak{R}}
\newcommand{\frS}{\mathfrak{S}}
\newcommand{\frT}{\mathfrak{T}}
\newcommand{\frU}{\mathfrak{U}}
\newcommand{\frV}{\mathfrak{V}}
\newcommand{\frW}{\mathfrak{W}}
\newcommand{\frX}{\mathfrak{X}}
\newcommand{\frY}{\mathfrak{Y}}
\newcommand{\frZ}{\mathfrak{Z}}
\newcommand{\fra}{\mathfrak{a}}
\newcommand{\frb}{\mathfrak{b}}
\newcommand{\frc}{\mathfrak{c}}
\newcommand{\frd}{\mathfrak{d}}
\newcommand{\fre}{\mathfrak{e}}
\newcommand{\frf}{\mathfrak{f}}
\newcommand{\frg}{\mathfrak{g}}
\newcommand{\frh}{\mathfrak{h}}
\newcommand{\fri}{\mathfrak{i}}
\newcommand{\frj}{\mathfrak{j}}
\newcommand{\frk}{\mathfrak{k}}
\newcommand{\frl}{\mathfrak{l}}
\newcommand{\frm}{\mathfrak{m}}
\newcommand{\frn}{\mathfrak{n}}
\newcommand{\fro}{\mathfrak{o}}
\newcommand{\frp}{\mathfrak{p}}
\newcommand{\frq}{\mathfrak{q}}
\newcommand{\frr}{\mathfrak{r}}
\newcommand{\frs}{\mathfrak{s}}
\newcommand{\frt}{\mathfrak{t}}
\newcommand{\fru}{\mathfrak{u}}
\newcommand{\frv}{\mathfrak{v}}
\newcommand{\frw}{\mathfrak{w}}
\newcommand{\frx}{\mathfrak{x}}
\newcommand{\fry}{\mathfrak{y}}
\newcommand{\frz}{\mathfrak{z}}
\newcommand{\frgl}{\mathfrak{gl}}
\newcommand{\frsl}{\mathfrak{sl}}
\newcommand{\frso}{\mathfrak{so}}
\newcommand{\frsp}{\mathfrak{sp}}
\newcommand{\frsu}{\mathfrak{su}}

\title[]{Toric varieties with ample tangent bundle}
\author{Kuang-Yu Wu}
\address{Kuang-Yu Wu, Department of Mathematics, Statistics, and Computer Science, University of Illinois at Chicago, Chicago, USA}
\email{kwu33@uic.edu}

\maketitle

\theoremstyle{plain}
\newtheorem*{thmA}{Theorem A}

\newcommand{\bfu}{\mathbf{u}}
\newcommand{\pair}[2]{\langle #1 , #2 \rangle}

\begin{abstract}
We give a simple combinatorial proof of the toric version of Mori's theorem
that the only $n$-dimensional smooth projective varieties with ample tangent bundle
are the projective spaces $\bbP^n$.
\end{abstract}

\section{Introduction}

It is a well-known theorem
that the only smooth projective varieties
(over an algebraically closed field $k$)
with ample tangent bundles
are the projective spaces $\bbP_k^n$.
This is first conjectured by Hartshorne \cite[Problem 2.3]{Har70}
and later proved by Mori \cite{Mor79}
using the full force of his now-celebrated ``bend and break'' technique.
Here we say that a vector bundle $\calE$ is ample (resp.\ nef)
if the line bundle $\calO_{\bbP\calE}(1)$ on the projectivized bundle $\bbP\calE$ is ample (resp.\ nef).

In this paper,
we consider a toric version of this theorem
and show that it admits a simple combinatorial proof.

\begin{thm} \label{thm}
	Let $X$ be an $n$-dimensional smooth projective toric variety (over an algebraically closed field $k$) with ample tangent bundle $\calT_X$.
	Then $X$ is isomorphic to $\mathbb{P}^n_k$.
\end{thm}

In the proof
we consider the polytope $P \subseteq \bbR^n$ corresponding to $X$
(together with any ample divisor $D$).
The key observation we make is that the ampleness of $\calT_X$ implies that
the sum of any pair of two adjacent angles on a $2$-dimensional face of $P$ is smaller than $\pi$.
It follows that $P$ has to be an $n$-simplex,
and hence $X$ is isomorphic to $\mathbb{P}^n$.


\subsection{Acknowledgment}

	The author was partially supported by the NSF grant DMS-1749447.
	I would like to thank my advisor Julius Ross for setting this project
	and for discussions about this project.

\section{Preliminaries}
{
	Here we list out some definitions and facts
	regarding toric varieties and toric vector bundles
	that we will use in this article.
	One may refer to \cite{Ful93,CLS11} for more details about toric varieties,
	and \cite{Pay08,DRJS18} for more details about toric vector bundles.
}
\subsection{Toric varieties}

We work throughout over an algebraically closed field $k$.
By a toric variety,
we mean an irreducible and normal algebraic variety $X$
containing a torus $T\cong(k^*)^n$ as a Zariski open subset
such that
the action of $T$ on itself (by multiplication)
extends to an algebraic action of $T$ on $X$.

Let $M$ be the group of the characters of $T$,
and $N$ the group of the $1$-parameter subgroups of $T$.
Both $M$ and $N$ are lattices of rank $n$
(equal to the dimension of $T$),
i.e. isomorphic to $\bbZ^n$.
They are dual to each other in the sense that
there is a natural pairing of $M$ and $N$ denoted by
$\langle\cdot,\cdot\rangle:M\times N\to\bbZ$.

Every toric variety $X$ is associated to a fan $\gS$ in $N_{\mathbb{R}}:=N\otimes_{\mathbb{Z}}\mathbb{R}\,(\cong\bbR^n)$.
A fan $\gS$ is said to be \textit{complete}
if it supports on the whole $N_{\bbR}$,
and is said to be \textit{smooth}
if every cone in $\gS$ is generated by
a subset of a $\bbZ$-basis of $N$.
A toric variety $X$ is complete
if and only if its associated fan $\gS$ is complete,
and $X$ is smooth if and only if $\gS$ is smooth.

There is an inclusion-reversing bijection
between the cones $\sigma\in\gS$
and the $T$-orbits in $X$.
Let $O_{\sigma} \subseteq X$ be the orbit corresponding to $\sigma$.
The codimension of $O_{\sigma}$ in $X$ is equal to
the dimension of $\sigma$.
Each cone $\gs \in \gS$ also corresponds to an open affine set $U_{\sigma} \in X$,
which is equal to the union of all the orbits $O_{\tau}$
corresponding to cones $\tau$ contained in $\sigma$.
Given a $1$-dimensional cone $\rho\in\gS$,
the closure of $O_{\rho}$ is a Weil divisor,
denoted by $D_{\rho}$.
The class group of $X$ is generated by the classes of the divisors $D_{\rho}$
corresponding to the $1$-dimensional cones in $\gS$.

\subsection{Polytopes and toric varieties}

Let $M_{\bbR}:=M\otimes_{\bbZ}\bbR\cong\bbR^n$.
A lattice polytope $P$ in $M_{\bbR}$ is
the convex hull of finitely many points in $M$.
The dimension of $P$ is the dimension of the affine span of $P$.
When $\dim P=\dim M_{\bbR}$,
we say that $P$ is full dimensional.

Let $P\subseteq M_{\bbR}$ be a full dimensional lattice polytope,
and let $P_1,...,P_m$ be the \textit{facets} of $P$,
i.e. codimension 1 faces of $P$.
For each facet $P_k$,
there exists a unique primitive lattice point $v_k\in N$
and a unique integer $c_k\in\bbZ$
such that
$$
P_k=\{u\in P\,|\,\langle u,v_k\rangle =-c_k\}
$$
and $\langle u,v_k\rangle\geq -c_k$ for all $u\in P$.

Define $\gS_P$ to be the complete fan
whose $1$-dimensional cones are exactly
those generated by $v_k$.
This fan $\gS_P$ is called the \textit{(inner) normal fan} of $P$.
The toric variety $X_{\gS_P}$ associated to $\gS_P$
is called the toric variety of $P$,
and denoted by $X_P$.
Denote by $D_k$
the divisor corresponding to the $1$-dimensional cone generated by $v_k$.
Then we may define a divisor on $X_P$ by $D_P:=\sum_{k=1}^mc_kD_k$.
Such a divisor $D_P$ is necessarily ample.

This process is reversible,
and there is a 1-to-1 correspondance
between full dimensional lattice polytope $P\subseteq M_{\bbR}$
and a pair $(X,D)$ of a complete toric variety $X$
together with an ample $T$-invariant divisor $D$ on $X$.


\subsection{Toric vector bundles}

A vector bundle $\pi:\mathcal{E}\to X$ over a toric variety $X=X_{\gS}$
is said to be toric (or equivariant) if there is a $T$-action on $\mathcal{E}$ such that $t\circ\pi=\pi\circ t$ for all $t\in T$.

Given a cone $\gs \in \gS$ and $u \in M$,
define the line bundle $\calL_u |_{U_\gs}$ over $U_{\gs}$ to be
the trivial line bundle $U_{\gs} \times k$
equipped with the $T$-action given by
$t.(x , z) := (t.x , \chi^u(t) \cdot z)$.
If $u , u' \in M$ satisfy $u - u' \in \gs^{\perp}$,
then $\chi^{u - u'}$ is a non-vanishing regular function on $U_{\gs}$
which gives an isomorphism $\calL_u |_{U_\gs} \cong \calL_{u'} |_{U_\gs}$.
In fact,
the group of toric line bundles on $U_{\gs}$ is isomorphic to $M_{\gs} := M / (M \cap \gs^{\perp})$.
Therefore,
we also write $\calL_{[u]} |_{U_\gs}$,
where $[u] \in M_{\gs}$ is the class of $u$.

Let $\calE\to X$ be a toric vector bundle of rank $r$.
Its restriction to an invariant open affine set $U_{\gs}$
splits into a direct sum of toric line bundles with trivial underlying line bundles \cite[Proposition 2.2]{Pay08};
i.e. we have
$\calE |_{U_{\gs}} \cong \bigoplus_{i = 1}^r \calL_{[u_i]} |_{U_\gs}$
for some $[u_i] \in M_{\gs}$.
Define the \textit{associated characters} of $\calE$ on $\gs$
to be the multiset $\bfu_{\calE}(\gs) \subset M_{\gs}$ of size $r$
that contains the $[u_i]$ showing up in the splitting.

\begin{exmp}[Associated characters of tangent bundles] \label{exmp:T_X}
	Let $X = X_{\gS}$ be an $n$-dimensional smooth projective toric variety,
	and consider its tangent bundle $\calT_X$.
	Fix a maximal cone $\gs \in \gS$.
	Since $X$ is smooth,
	the dual cone $\check{\gs}$ of $\gs$
	is generated by some $u_1 , ... , u_n \in M$ that form a $\bbZ$-basis of $M$.
	Denote by $x_1 , ... , x_n \in \gC(U_{\gs} , \calO_X)$ the coordinates on $U_{\gs} \cong k^n$ corresponding to $u_1 , ... , u_n$.
	Then $\big\{ \pdv{x_1} , ... , \pdv{x_n} \big\}$ is a local frame of $\calT_X$ on $U_{\gs}$.
	Each non-vanishing section $\pdv{x_i} \in \gC(U_{\gs} , \calT_X)$ naturally generates a toric line bundle on $U_{\gs}$ isomorphic to $\calL_{u_i}|_{U_{\gs}}$.
	Thus we have
	$
	\calT_X |_{U_{\gs}} \cong \bigoplus_{i = 1}^n \calL_{u_i} |_{U_\gs} \, ,
	$
	and hence the associated characters of $\calT_X$ on $\gs$ are $\bfu_{\calT_X}(\gs) = \{u_1 , ... , u_n\}$.
\end{exmp}

\subsection{Positivity of toric vector bundles} \label{sec:pos}

Let $X = X_{\gS}$ be a complete toric variety.
By an \textit{invariant curve} on $X$,
we mean a complete irreducible $1$-dimensional subvariety
that is invariant under the $T$-action.
Via the cone-orbit correspondence,
there is a one-to-one correspondence between the invariant curves and the codimension-$1$ cones;
every invariant curve is the closure of an $1$-dimensional orbit,
which corresponds to a codimension-$1$ cone in $\gS$.
For each codimension-$1$ cone $\tau \in \gS$,
denote the corresponding invariant curve by $C_{\tau}$.

The positivity of toric vector bundles can be checked on invariant curves according to the following result in \cite{HMP10}.

\begin{thm}\cite[Theorem 2.1]{HMP10} \label{thm:HMP}
	A toric vector bundle on a complete toric variety is ample (resp.\ nef)
	if and only if
	its restriction to every invariant curve is ample  (resp.\ nef).
\end{thm}

Note that every invariant curve is a $\bbP^1$.
By Birkhoff-Grothendieck theorem,
every vector bundle on $\bbP^1$ splits into a direct sum of line bundles.
Hence,
the positivity of vector bundles on $\bbP^1$ is well understood,
namely $\bigoplus_{i = 1}^r \calO_{\bbP^1} (a_i)$ is ample (resp.\ nef)
if and only if every $a_i$ is positive (resp.\ non-negative).
It is common to call the $r$-tuple (or multiset) $(a_i)_{i = 1}^r$ the \textit{splitting type} of the vector bundle.

Fix a codimension-$1$ cone $\tau$,
and let $\gs , \gs'$ be the two maximal cones containing $\tau$.
Given $u , u' \in M$ satisfying $u - u' \in \tau^{\perp}$,
define a toric line bundle $\calL_{u , u'}$ on $U_{\gs} \cup U_{\gs'}$
by glueing the toric line bundles $\calL_{u}|_{U_{\gs}}$ and $\calL_{u'}|_{U_{\gs'}}$
with the transition function $\chi^{u' - u}$.
Since the invariant curve $C_{\tau}$ is contained in $U_{\gs} \cup U_{\gs'}$,
we may restrict $\calL_{u , u'}$ to get a toric line bundle $\calL_{u , u'}|_{C_{\tau}}$ on $C_{\tau}$.

\begin{prop}
	\cite[Corollary 5.5 and 5.10]{HMP10} \label{prop:split}
	Let $X$ be a complete toric variety.
	Any toric vector bundle $\mathcal{E}|_{C_{\tau}}$ on the invariant curve $C_{\tau}$ splits equivariantly as a sum of line bundles
	$$
	\mathcal{E}|_{C_{\tau}}=\bigoplus_{i=1}^r\mathcal{L}_{u_i,u_i'}|_{C_{\tau}}.
	$$
	The splitting is unique up to reordering.
\end{prop}

Combining this with the following lemma that computes the underlying line bundle of $\calL_{u , u'}|_{C_{\tau}}$,
one gets the splitting type of $\calE|_{C_{\tau}}$.

\begin{lem} \cite[Example 5.1]{HMP10} \label{lem:O(m)}
	Let $u_0$ be the generator of $M \cap \tau^{\perp} \cong \bbZ$
	that is positive on $\gs$,
	and let $m$ be the integer such that $u - u' = m u_0$.
	Then, the underlying line bundle of
	$\calL_{u , u'}|_{C_{\tau}}$ is isomorphic to $\calO_{\bbP^1} (m)$.
\end{lem}


\section{Restricting $\calT_X$ to invariant curves}
{
Let $X = X_{\gS}$ be a smooth complete toric variety of dimension $n$.
In this section,
we consider the restrictions of the tangent bundle $\calT_X$ to the invariant curves.
The goal is to get the splitting types 
in terms of the combinatorial data of the fan $\gS$ of $X$.
This has in fact been done in \cite[Example 5.1 and 5.2]{DRJS18} and \cite[Theorem 2]{Sch18}.
We repeat the calculation for the convenience of the readers.
}

{
Fix an $(n-1)$-dimensional cone $\tau \in \gS$.
Let $\gs , \gs' \in \gS(n)$ be the two maximal cones containing $\tau$.
Let $v_1 , ... , v_{n - 1} , v_n , v'_n \in N$ be primitive vectors
such that $\tau$ is generated by $\{ v_1 , ... , v_{n - 1} \}$,
$\gs$ is generated by $\{ v_1 , ... , v_{n - 1} , v_n \}$,
and $\gs'$ is generated by $\{ v_1 , ... , v_{n - 1} , v'_n \}$.
There are unique $u_i , u'_i \in M$ ($i = 1 , ... , n$)
such that
$\pair{u_i}{v_i} = \pair{u'_i}{v'_i} = 1$ for all $i$
and $\pair{u_i}{v_j} = \pair{u'_i}{v'_j} = 0$ for all $i \neq j$,
where we define $v'_i = v_i$ for $i = 1 , ... , n - 1$.
The dual cones $\check{\gs}$ and $\check{\gs}'$
are generated by $\{ u_1 , ... , u_n \}$ and $\{ u'_1 , ... , u'_n \}$, respectively.

By Example \ref{exmp:T_X},
the associated characters of $\calT_X$ on $\gs$ and $\gs'$ are given by
$$
\bfu_{\calT_X}(\gs) = \{ u_1 , ... , u_n \}
\; , \;
\bfu_{\calT_X}(\gs') = \{ u'_1 , ... , u'_n \}
\; .
$$
Following Section \ref{sec:pos},
let $C_{\tau}$ be the invariant curve corresponding to $\tau$.
The splitting of $\calT_X|_{C_{\tau}}$ as in Proposition \ref{prop:split} is easy to get by the following fact.
}

\begin{lem} \label{lem}
	We have $u_i - u'_i \in \tau^{\perp}$
	for all $i = 1 , ... , n$,
	and $u_i - u'_j \notin \tau^{\perp}$
	for all $i \neq j$.
\end{lem}

\begin{proof}
	The first part follows from
	$\pair{u_i - u'_i}{v_{i'}} = 0$ for all $i' = 1 , ... , n - 1$,
	and the second part follows from
	$\pair{u_i - u'_j}{v_i} = -\pair{u_i - u'_j}{v_j} = 1$,
	where at least one of $i , j$ is not $n$.
\end{proof}

\begin{defn}
	Define $a_i \in \bbZ$ ($i = 1 , ... , n$) to be the integers satisfying $u_i = u'_i + a_i u_n$.
	Such integers exist since $u_n$ is a primitive generator of $\tau^{\perp} \cap M \cong \bbZ$.
	Note that $u'_n = - u_n$ so that $a_n = 2$.
\end{defn}

\begin{prop} \label{prop3}
	On the invariant curve $C_{\tau}$,
	the restriction $\calT_X |_{C_{\tau}}$ of the tangent bundle (as a toric vector bundle) splits into the following direct sum of toric line bundles
	$$
	\calT_X |_{C_{\tau}}
	\cong
	\bigoplus_{i = 1}^n \calL_{u_i , u'_i}|_{C_{\tau}} \, .
	$$
	In particular,
	we have the following splitting of $\calT_X |_{C_{\tau}}$ as a vector bundle
	$$
	\calT_X |_{C_{\tau}}
	\cong
	\bigoplus_{i = 1}^n \calO_{\bbP^1} (a_i) \, .
	$$
\end{prop}

\begin{proof}
	By Proposition \ref{prop:split},
	$\calT_X |_{C_{\tau}}$ splits into a direct sum of toric line bundles of the form $\calL_{u,u'}|_{C_{\tau}}$.
	This gives a bijection
	$\gi : \bfu_{\calE}(\gs) \to \bfu_{\calE}(\gs')$
	mapping $u$ to $u'$ whenever $\calL_{u,u'}|_{C_{\tau}}$ shows up in the splitting.
	Note that $u_i - \gi (u_i) \in \tau^{\perp}$ by the definiton of $\calL_{u,u'}$.
	Then Lemma \ref{lem} implies that we must have $\gi (u_i) = u'_i$ for all $i$,
	hence the splitting in the first part.
	
	The second part follows directly from the first part together with Lemma \ref{lem:O(m)}.
\end{proof}

\begin{rmk}
	The integers $a_i$ are the same as the integers $b_i$
	that show up in the ``wall relation''
	$$
	b_1 v_1 + \cdots + b_{n - 1} v_{n - 1} + v_n + v'_n = 0 \, ,
	$$
	mentioned in \cite{Sch18} and \cite{DRJS18}.
	Indeed we have $b_i = - \pair{u_i}{v'_n} = a_i$ for all $i$.
\end{rmk}

\begin{exmp} \label{exmp:fans}
	For each of the following toric surfaces $X$,
	we fix a $1$-dimensional cone $\tau$ in its fan
	(as shown in Figure \ref{fig:fans})
	and compute the splitting type of $\calT_X|_{C_{\tau}}$.
	\begin{enumerate}
		\item[(1)]
		$X = \bbP^2$.
		The dual cones of the maximal cones containing $\tau$ are given by
		$\check{\sigma} = \operatorname{Cone} \{(-1 , 0) , (-1 , 1)\}$
		and
		$\check{\sigma} = \operatorname{Cone} \{(0 , -1) , (1 , -1)\}$.
		Therefore we get $\calT_X|_{C_{\tau}} \cong \calO_{\bbP^1} (1) \oplus \calO_{\bbP^1} (2)$.
		In fact,
		the restrictions of $\calT_X$ to the other two invariant curves have the same splitting type,
		so $\calT_X$ is ample by Proposition \ref{thm:HMP}.
		
		\item[(2)]
		$X = \bbP^1 \times \bbP^1$.
		The dual cones of the maximal cones containing $\tau$ are given by
		$\check{\sigma} = \operatorname{Cone} \{(-1 , 0) , (0 , 1)\}$
		and
		$\check{\sigma} = \operatorname{Cone} \{(-1 , 0) , (0 , -1)\}$.
		Therefore we get $\calT_X|_{C_{\tau}} \cong \calO_{\bbP^1} (0) \oplus \calO_{\bbP^1} (2)$.
		In fact,
		the restrictions of $\calT_X$ to the other three invariant curves have the same splitting type,
		so $\calT_X$ is nef (but not ample) by Proposition \ref{thm:HMP}.
		
		\item[(3)]
		Let $X$ be the Hirzebruch surface $\bbF_1$,
		which is isomorphic to $\bbP^2$ blown up at one point.
		The dual cones of the maximal cones containing $\tau$ are given by
		$\check{\sigma} = \operatorname{Cone} \{(-1 , 0) , (0 , 1)\}$
		and
		$\check{\sigma} = \operatorname{Cone} \{(-1 , 1) , (0 , -1)\}$.
		Therefore we get $\calT_X|_{C_{\tau}} \cong \calO_{\bbP^1} (-1) \oplus \calO_{\bbP^1} (2)$,
		and hence $\calT_X$ is not nef by Proposition \ref{thm:HMP}.
	\end{enumerate}
\end{exmp}

\setcounter{figure}{\value{thm}}
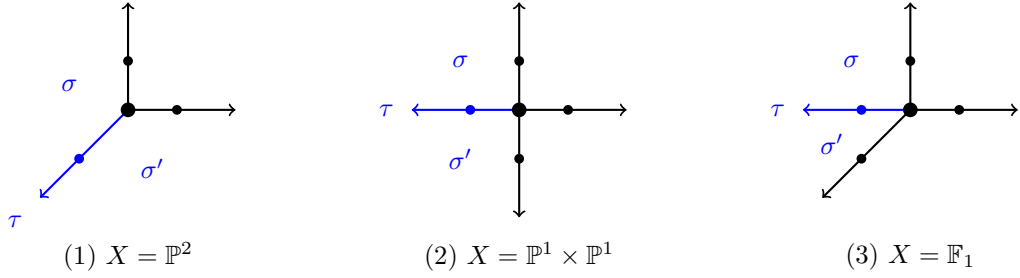
\begin{figure}[h!]
	\begin{tikzpicture}[scale=0.65]
		\draw (-8 , -3) node {(1) $X = \bbP^2$} ;
		\draw[thick , ->] (-8 , 0) -- (-5.8 , 0) ;
		\fill (-7 , 0) circle (1mm) ;
		\draw[thick , ->] (-8 , 0) -- (-8 , 2.2) ;
		\fill (-8 , 1) circle (1mm) ;
		\draw[thick , blue , ->] (-8 , 0) -- (-9.8 , -1.8) node[below left=3pt] {$\tau$} ;
		\fill[blue] (-9 , -1) circle (1mm) ;
		\draw[blue] (-9.2 , 0.5) node {$\gs$} ;
		\draw[blue] (-7.5 , -1.2) node {$\gs'$} ;
		\fill (-8 , 0) circle (1.5mm) ;
		
		\draw (0 , -3) node {(2) $X = \bbP^1 \times \bbP^1$} ;
		\draw[thick , ->] (0 , 0) -- (2.2 , 0) ;
		\fill (1 , 0) circle (1mm) ;
		\draw[thick , ->] (0 , 0) -- (0 , 2.2) ;
		\fill (0 , 1) circle (1mm) ;
		\draw[thick , ->] (0 , 0) -- (0 , -2.2) ;
		\fill (0 , -1) circle (1mm) ;
		\draw[thick , blue , ->] (0 , 0) -- (-2.2 , 0) node[left=3pt] {$\tau$} ;
		\fill[blue] (-1 , 0) circle (1mm) ;
		\draw[blue] (-1.2 , 1) node {$\gs$} ;
		\draw[blue] (-1.2 , -1) node {$\gs'$} ;
		\fill (0 , 0) circle (1.5mm) ;

		\draw (8 , -3) node {(3) $X = \bbF_1$} ;
		\draw[thick , ->] (8 , 0) -- (10.2 , 0) ;
		\fill (9 , 0) circle (1mm) ;
		\draw[thick , ->] (8 , 0) -- (8 , 2.2) ;
		\fill (8 , 1) circle (1mm) ;
		\draw[thick , ->] (8 , 0) -- (6.2 , -1.8) ;
		\fill (7 , -1) circle (1mm) ;
		\draw[thick , blue , ->] (8 , 0) -- (5.8 , 0) node[left=3pt] {$\tau$} ;
		\fill[blue] (7 , 0) circle (1mm) ;
		\draw[blue] (6.8 , 1) node {$\gs$} ;
		\draw[blue] (6.4 , -0.7) node {$\gs'$} ;
		\fill (8 , 0) circle (1.5mm) ;

	\end{tikzpicture}
	\caption{Fans of toric surfaces}
	\label{fig:fans}
\end{figure}

\section{Polytopes and ampleness of the tangent bundle}

Let $X = X_{\gS}$, $\calT_X$, $\tau$, $\gs$, $\gs'$, $u_i$, $u'_i$, $a_i$ be as in the previous section.

Fix an ample $T$-invariant divisor $D$,
and let $P = P(X , D)$ be the corresponding polytope.
Note that $X$ and $\gS$ are simplicial as they are smooth;
in particular,
every maximal cone in $\gS$ has exactly $n$ faces of dimension $(n - 1)$,
and every $(n - 1)$-dimensional cone has exactly $(n - 1)$ faces of dimension $(n - 2)$.
This implies that there are exactly $n$ edges emanating from every vertex of $P$
and that every edge of $P$ is contained in exactly $(n - 1)$ faces of dimension $2$.

Let $p_{\gs} \in P$ be the vertex corresponding to the maximal cone $\gs$.
Denote by $P - p_{\gs}$ the translation of $P$ by $- p_{\gs}$.
The cone generated by $P - p_{\gs}$ is given by
$\{ u \in M_{\bbR} \, | \, \pair{u}{v_i} \geq 0 \text{ for all } i = 1 , ... , n \}$,
which is exactly the dual cone $\check{\gs}$ of $\gs$.
Thus,
the $n$ edges of $P$ emanating from $p_{\gs}$ are parallel to $u_1 , ... , u_n$.
Similarly the $n$ edges emanating from the vertex $p_{\gs'}$ corresponding to $\gs'$ are parallel to $u'_1 , ... , u'_n$.

Recall that the $u_i$ and $u'_i$ satisfy
$u'_i = u_i - a_i u_n$ for all $i = 1 , ... , n - 1$
and $u'_n = - u_n$.
Since $\gs$ and $\gs'$ contain the $(n - 1)$-dimensional cone $\tau$ as a common face,
the convex hull of $\overline{p_{\gs},p_{\gs'}}$ of $p_{\gs}$ and $p_{\gs'}$ is an edge of $P$;
it corresponds to $\tau$ and is parallel to $u_n$ and $u'_n$.
Fix a $j \in \{ 1 , ... , n - 1 \}$.
Consider the points $p_{\gs} + u_j , p_{\gs'} + u'_j \in M$.
The point $p_{\gs} + u_j$ is on an edge emanating from $p_{\gs}$,
and $p_{\gs'} + u'_j$ is on an edge emanating from $p_{\gs'}$.
In addition,
since $(p_{\gs} + u_j) - (p_{\gs'} + u'_j) = (p_{\gs} - p_{\gs'}) + m_j u_n$,
$\overline{p_{\gs} + u_j , p_{\gs'} + u'_j}$ is parallel to $\overline{p_{\gs},p_{\gs'}}$.
Thus,
the four points $p_{\gs},p_{\gs'},p_{\gs} + u_i,p_{\gs'} + u'_i$
are contained in a common $2$-dimensional face $A_i \subseteq P$.
In fact,
$A_i$ is the $2$-dimensional face of $P$
corresponding to the $(n - 2)$-dimensional cone
$\tau \cap (u_i)^{\perp} = \tau \cap (u'_i)^{\perp}$.

Denote the angles at $p_{\sigma}$ and $p_{\sigma'}$ on $A_j$ by $\gth (p_{\gs} , A_j)$ and $\gth (p_{\gs'} , A_j)$, respectively.
Their sum is related to the integer $a_j$ in the following way.

\begin{prop} \label{prop4}
	The sum $\theta (p_{\sigma} , A_j) + \theta (p_{\sigma'} , A_j)$ is smaller than $\pi$ if and only if $a_j > 0$, equal to $\pi$ if and only if $a_j = 0$, and greater than $\pi$ if and only if $a_j < 0$.
\end{prop}

\begin{proof}
	Suppose $a_j < 0$.
	Consider the quadrilateral with vertices
	$p_{\gs} , p_{\gs'} , p_{\gs'} + u'_j , p_{\gs} + u_j$.
	It is a trapezoid with the edges $\overline{p_{\gs} + u_j , p_{\gs'} + u'_j}$ and $\overline{p_{\sigma},p_{\sigma'}}$
	parallel to each other.
	Since
	$$
	\big( (p_{\gs'} + u'_j) - (p_{\gs} + u_j) \big) - (p_{\gs'} - p_{\gs}) = - a_j u_1 ,
	$$
	the edge $\overline{p_{\sigma}+u_i,p_{\sigma'}+u_i'}$ is longer than $\overline{p_{\sigma},p_{\sigma'}}$,
	implying $\gth(p_{\gs} , A_j) + \gth(p_{\gs'} , A_j) > \pi$.
	
	The other two cases are similar.
\end{proof}

\begin{rmk}
	Although the angles $\gth (p_{\gs} , A_j) , \gth (p_{\gs'} , A_j)$ themselves are not invariant under a change of bases of $M$,
	whether their sum is smaller than, equal to, or greater than $\pi$ is.
\end{rmk}

\begin{exmp}
	In Figure \ref{fig:polytopes} are polytopes $P (X , -K_X)$ corresponding to the toric surfaces $X$ in Example \ref{exmp:fans} together with their anticanonical line bundles $-K_X$.
	The cones $\tau , \gs , \gs'$ are the same as in Example \ref{exmp:fans}.
	\begin{enumerate}
		\item[(1)]
		$X = \bbP^2$.
		Recall $\calT_X|_{C_{\tau}} \cong \calO_{\bbP^1} (1) \oplus \calO_{\bbP^1} (2)$ so that $a_1 = 1 > 0$.
		Here we see $\theta (p_{\sigma} , P) + \theta (p_{\sigma'} , P) < \pi$.
		
		\item[(2)]
		$X = \bbP^1 \times \bbP^1$.
		Recall $\calT_X|_{C_{\tau}} \cong \calO_{\bbP^1} (0) \oplus \calO_{\bbP^1} (2)$ so that $a_1 = 0$.
		Here we see $\theta (p_{\sigma} , P) + \theta (p_{\sigma'} , P) = \pi$.
		
		\item[(3)]
		$X = \bbF_1$.
		Recall $\calT_X|_{C_{\tau}} \cong \calO_{\bbP^1} (-1) \oplus \calO_{\bbP^1} (2)$ so that $a_1 = -1 < 0$.
		Here we see $\theta (p_{\sigma} , P) + \theta (p_{\sigma'} , P) > \pi$.
	\end{enumerate}
\end{exmp}

\setcounter{figure}{\value{thm}}
\begin{figure}[h!]
	\begin{tikzpicture}[scale=0.65]
		\draw (-8 , -3) node {(1) $X = \bbP^2$} ;
		\draw[very thin , color=lightgray] (-10 , -2) grid (-6 , 2) ;
		\draw[thick] (-9 , 2) -- (-9 , -1) -- (-6 , -1) ;
		\draw[thick , blue] (-6 , -1) -- (-9 , 2) ;
		\fill[blue]
			(-9 , 2) circle (1mm) node[above=2pt] {$p_{\gs'}$}
			(-6 , -1) circle (1mm) node[right=3pt] {$p_{\gs}$}
			;
		\fill (-8 , 0) circle (1.5mm) ;
		\fill (-9 , -1) circle (1mm) ;
		
		\draw (0 , -3) node {(2) $X = \bbP^1 \times \bbP^1$} ;
		\draw[very thin , color=lightgray] (-2 , -2) grid (2 , 2) ;
		\draw[thick] (1 , 1) -- (-1 , 1) -- (-1 , -1) -- (1 , -1) ;	
		\draw[thick , blue] (1 , -1) -- (1 , 1) ;
		\fill[blue]
			(1 , 1) circle (1mm) node[above right] {$p_{\gs'}$}
			(1 , -1) circle (1mm) node[below right] {$p_{\gs}$}
			;
		\fill (0 , 0) circle (1.5mm) ;
		\fill
			(-1 , 1) circle (1mm)
			(-1 , -1) circle (1mm)
			;
		
		\draw (8 , -3) node {(3) $X = \bbF_1$} ;
		\draw[very thin , color=lightgray] (6 , -2) grid (10 , 2) ;
		\draw[thick] (9 , 0) -- (7 , 2) -- (7 , -1) -- (9 , -1) ;
		\draw[thick , blue] (9 , -1) -- (9 , 0) ;
		\fill[blue]
			(9 , 0) circle (1mm) node[above right] {$p_{\gs'}$}
			(9 , -1) circle (1mm) node[below right] {$p_{\gs}$}
			;
		\fill (8 , 0) circle (1.5mm) ;
		\fill
			(7 , 2) circle (1mm)
			(7 , -1) circle (1mm)
			;
	\end{tikzpicture}
	\caption{Polytopes $P(X , -K_X)$ of toric surfaces}
	\label{fig:polytopes}
\end{figure}
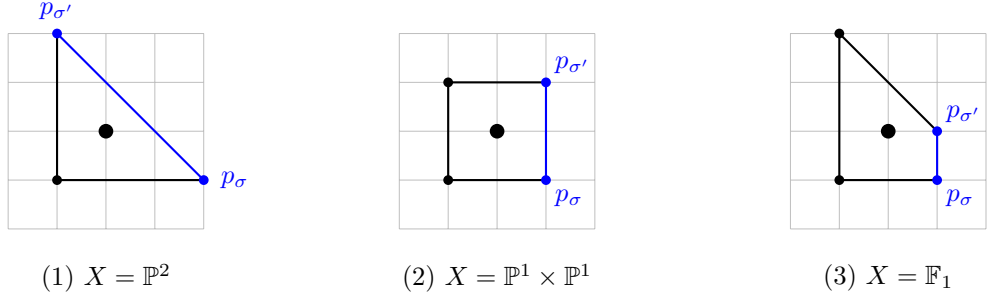

\section{Proof of Theorem \ref{thm}}


\begin{proof}[Proof of Theorem \ref{thm}]
	As promised,
	we will show that the polytope $P$ corresponding to $X$
	(together with any ample $T$-invariant divisor $D$) is an $n$-simplex.
	
	Let $A$ be a $2$-dimensional face of $P$.
	Let $m$ be the number of vertices of $A$,
	and let $p_1,...,p_m$ be the vertices of $A$,
	ordered so that $p_k$ is adjacent to $p_{k+1}$ for all $k = 1 , ... , m$ (where $p_{m+1}:=p_1$).
	Since $\calT_X$ is ample,
	its restriction to every invariant curve is ample.
	Then, by Proposition \ref{prop3} and Proposition \ref{prop4},
	$\gth(p_k , A) + \gth(p_{k+1} , A) < \pi$ for all $k$.
	This implies
	$$
	m \pi > \sum_{k = 1}^m (\gth (p_k , A) + \gth (p_{k+1} , A)) = 2 \sum_{k = 1}^m \gth (p_k , A) = 2 (m - 2) \pi.
	$$
	We get $m < 4$,
	implying $A$ is a triangle.
	The same is true for all $2$-dimensional faces of $P$.
	
	Now, we start with a vertex $q_0$ of $P$.
	Recall that every vertex of $P$ is adjacent to exactly $n$ vertices since $X$ is smooth and hence simplicial.
	Let $q_1,...,q_n$ be the $n$ points adjacent to $q_0$.
	Given $1<j\leq n$,
	let $A_j$ be the $2$-dimensional face containing the edges $\overline{q_0q_1}$ and $\overline{q_0q_j}$.
	Since $A_j$ is in fact a triangle,
	$q_1$ is also adjacent to $q_j$.
	Thus $q_1$ is adjacent to $q_0 , q_2 , ... , q_n$.
	Similarly, every $p_j$ is adjacent to exactly $p_0,...,\widehat{p_j},...,p_n$.
	Consequently, $p_0 , p_1 , ... , p_n$ are the only vertices of $P$,
	and hence $P$ is the $n$-simplex with vertices $p_0 , p_1 , ... , p_n$.
\end{proof}

\end{document}